\documentclass[12pt,psamsfonts,twoside,leqno]{article}

\usepackage{amsbsy,hyperref,amsmath,amsthm,amsfonts,graphicx,amssymb,mathtools}
\usepackage[all,arc]{xy}
\usepackage{enumerate}
\usepackage{mathrsfs}
\usepackage{breqn}

\newtheorem{thm}{Theorem}[section]

\newtheorem{lem}[thm]{Lemma}

\newtheorem{rem}[thm]{Remark}

\theoremstyle{definition}
\newtheorem{defn}[thm]{Definition}

\theoremstyle{remark}

\makeatletter
\let\c@equation\c@thm
\makeatother
\numberwithin{equation}{section}

\everymath=\expandafter{\the\everymath\displaystyle}

\usepackage{titlesec}

\titleformat{\section}
  {\normalfont\fontfamily{cmr}\fontsize{11}{12}\bfseries}{\thesection}{1em}{}

\numberwithin{equation}{section}
\begin{document}

\baselineskip=17pt




\title{Asymptotic nature of higher Mahler Measure}


\author{Arunabha Biswas\\
Department of Mathematics and Statistics\\
Texas Tech University\\
Broadway \& Boston, Lubbock, TX 79409-1042, U.S.A.\\
E-mail: arunabha.biswas@ttu.edu}

\date{}

\maketitle

\renewcommand{\thefootnote}{}



\renewcommand{\thefootnote}{\arabic{footnote}}
\setcounter{footnote}{0}

\begin{abstract} 
We consider Akatsuka's zeta Mahler measure as a generating function of higher Mahler measure $\displaystyle m_k(P)$ of a polynomial $P,$ where $\displaystyle m_k(P)$ is the integral of $\displaystyle \log^{k}\left|  P  \right|$ over the complex unit circle.  Restricting ourselves to $\displaystyle P(x)=x-r$ with $\displaystyle \left| r \right|=1$ we show some new asymptotic results regarding $\displaystyle m_k(P)$, especially $\displaystyle \frac{\left| m_k(P) \right|}{k!} \rightarrow \frac{1}{\pi}$ as $\displaystyle k \rightarrow \infty.$
\end{abstract}

\section{Introduction}

\begin{defn}  Given a non-zero Laurent polynomial $\displaystyle P(x) \in \mathbb{C}[x^{\pm 1}] $ and $\displaystyle k \in \mathbb{N}$, the $k$-higher Mahler measure of $P$ (see \cite{R2}) is defined by

$$\displaystyle m_{k}(P):=\int\limits_{0}^{1} \log^{k}\left| P \left( e^{2 \pi i \theta} \right) \right|d\theta=\frac{1}{2\pi i}\int \limits_{|z|=1} \log^k|P(z)|\frac{dz}{z}.$$
\end{defn}

These $\displaystyle m_k$'s are multiples of the coefficients in the Taylor expansion of Akatsuka's zeta Mahler measure (see \cite{R1}) $$\displaystyle Z(s,P):=\int \limits^{1}_{0} \left| P \left( e^{2 \pi i \theta} \right) \right|^{s}d\theta, \hspace{.5cm} \text{that is,} \hspace{.5cm} Z(s,P)= \sum \limits_{k=0}^{\infty} \frac{m_{k}(P)}{k!}s^{k}.$$ 

\noindent For $k=0,1,2,\, \cdots,$ let $\displaystyle a_{k}(P)=m_k(P)/k!,$ so that $$\displaystyle Z(s,P)=\sum \limits_{k=0}^{\infty} a_{k}(P)s^{k}.$$   \\

In this paper we only consider polynomials of type $P(x)=x-r$ with $\left| r \right| =1$. Therefore, from now on, we use the notations $m_{k}\left( x-r \right)=m_{k}$ and $a_{k}\left( x-r \right)=a_{k}$ for simplicity.

\section{Asymptotic nature of higher Mahler measure of $(r-x)$ when $\left| r \right|=1$}
\begin{thm} Let $\displaystyle m_k$ and $\displaystyle a_k$ be as above. Then \\
\label{eq:th}

\newtheorem{theorem}[subsubsection]{Theorem} 

\begin{enumerate}[\hspace{.5cm} (a)]
\item
$
\frac{m_{k+1}}{(k+1)!}+\frac{m_k}{k!}=a_{k+1}+a_k=\mathcal{O}\left( 1/k \right),
\label{eq:O_1_k}
$

\vspace{.5cm}

\item
$
\lim \limits_{k \to \infty} \left| \frac{m_k}{k!} \right|=\lim \limits_{k \to \infty} \left| a_k \right| =\frac{1}{\pi},
\label{eq:1_pi}
$

\vspace{.5cm}

\item
$
\frac{m_{k+1}}{(k+1)!}+\frac{m_k}{k!}=a_{k+1}+a_k=o\left( 1/k \right),
\label{eq:o_1_k}
$

\vspace{.5cm}

\item
$
\lim \limits_{k \to \infty} \frac{1}{k+1}\cdot\frac{m_{k+1}}{m_{k}}=\lim \limits_{k \to \infty} \frac{a_{k+1}}{a_k}=-1.
\label{eq:-1}
$

\end{enumerate}
\end{thm}

\vspace{.5cm}

From \cite{R2} we know that for $\left| s \right| < 1$,
\begin{eqnarray}
Z(s,r-x)= \exp \left( \sum \limits_{k=2}^{\infty} \frac{(-1)^{k}\left(1-2^{1-k}\right)\zeta(k)}{k}s^{k} \right).
\label{eq:z}
\end{eqnarray}
Therefore differentiating both sides of \eqref{eq:z} with respect to $s$ we obtain

 $$\setlength\arraycolsep{0.3em}
 \begin{array}{lclcl}

		&\,& \sum \limits_{k=1}^{\infty} k\, a_{k}\, s^{k-1} &=& \frac{\partial}{\partial s} Z(s,r-x) \\[.7cm]

                &\,& & = & Z(s,r-x)\sum \limits_{k=2}^{\infty} (-1)^{k}\left(1-2^{1-k}\right)\zeta(k)\, s^{k-1}\\[.7cm]

		&\,& &=& \left(\sum \limits_{k=0}^{\infty} a_{k}s^{k}\right) \left(\sum \limits_{k=1}^{\infty} b_{k}\, s^{k}\right)\\[.7cm]

		&\,& &=& \sum \limits_{k=1}^{\infty} \left( a_{0}\, b_{k} + \sum \limits_{j=1}^{k-1}\,a_{j}\, b_{k-j} \right)\,s^{k}\\[.7cm]
\end{array}
$$ 
where $\displaystyle b_{k-1}:=(-1)^{k}\, \left( 1-2^{1-k} \right)\, \zeta(k).$ From the power series expansion of \eqref{eq:z} we already know that $a_{0}=1$. Now comparing coefficients on both the sides of the last expression we get $a_{1}=0$, $\displaystyle a_{2}=\frac{1}{2}\,a_{0}\,b_{1}=\frac{1}{4}\zeta(2)$ and for $k\geq 3$
\begin{eqnarray}
a_{k}=\frac{1}{k}\, \sum \limits_{j=0}^{k-2}a_{j}\, b_{k-1-j}\, ,  
\label{eq:a_k}
\end{eqnarray} 
\noindent where 
\begin{eqnarray}
b_{k}:=(-1)^{k+1}\, \left( 1-2^{-k} \right)\, \zeta(k+1).
\label{eq:b_k}
\end{eqnarray}

\vspace{.5cm}

\section{A few required remarks and lemmas}

\begin{rem}
It can be easily shown by induction that $a_{2k}>0$ and $a_{2k+1}<0$ for all $k\geq 1.$ It is also easy to see that $ a_k =\frac{(-1)^k}{k}\, \sum \limits_{j=0}^{k-2} \left| a_j b_{k-1-j} \right|$ for $k>1.$
\label{eq:b_alternating}
\end{rem}

\begin{rem}
Let $ B_k:=\left|b_{k} \right|$. Then $B_k \leq 1$ for all $k \geq 1,$ $B_k$ is increasing and $B_k \rightarrow 1$ as $k \rightarrow \infty.$
\label{eq:b_les_1}
\end{rem}

\noindent Notice $B_{k}=\eta(k+1)$ where $\eta(k)$ is Dirichlet's eta function. Since $\eta(k) \rightarrow 1$ as $k \rightarrow \infty$ and $\eta(k)$ is an increasing function of $k$ by \cite{eta_increasing}, $B(k) \leq1$ for all $k \geq 1,$ $B_k$ is increasing and $B_k \rightarrow 1$ as $k \rightarrow \infty.$

\begin{lem}
$\left| a_{k} \right| \leq 1$ for all $k \geq 1.$
\label{eq:a_les_1}
\end{lem}

\begin{proof}
We use induction to prove this. First we see that $\left| a_{0} \right|=1\leq 1$, $\left| a_{1} \right|=0\leq 1$, and $\left| a_{2} \right|=\zeta(2)/4 =\pi^{2}/24 \leq 1.$ Now let us assume $\left| a_{j} \right| \leq 1$ for all $2<j<k$. Using this along with Remark \ref{eq:b_les_1} we get
$$\left| a_{k} \right|=\frac{1}{k}\left| \sum \limits_{j=0}^{k-2}\, a_{j}\, b_{k-1-j} \right|\leq \frac{1}{k}\,\sum \limits_{j=0}^{k-2}\, \left| a_{j}\, b_{k-1-j} \right|\leq \frac{1}{k}\, \sum \limits_{j=0}^{k-2} \, 1 =\frac{k-1}{k}<1.$$
\end{proof}

\begin{lem}
For $k\geq 4,$ $\zeta(k)-\zeta(k+1)\leq \frac{1}{k^2}.$
\label{eq:diff_zeta_k_les}
\end{lem}

\begin{proof}
We use induction to prove it. But first notice that for all $k\geq 4$ and $n\geq 2$ we have $0<\frac{\sqrt{n}}{\sqrt{n}-1}< 4 \leq k,$ from which it follows that $n\left( 1-\frac{1}{k} \right)^2\geq 1.$
For $k=4$ we see that $\zeta(4)-\zeta(5)\approx 0.045<0.0625=\frac{1}{4^2}.$ Assume the conclusion of the lemma is true for all $4< j < k,$ that means we assume it is true for $j=k-1.$ Since for all $k\geq 4$ and $n\geq 2$ we have $n\left( 1-\frac{1}{k} \right)^2\geq 1$, therefore
$$\setlength\arraycolsep{0.3em}
\begin{array}{lclcl}
&\,& \frac{1}{k^2} &=& \left( \frac{k-1}{k} \right)^{2}\cdot \frac{1}{\left( k-1 \right)^{2}}  \\[.5cm]
 &\,&  &\geq& \left( 1-\frac{1}{k}\right)^{2}\left(\zeta(k-1)-\zeta(k)\right)  \\[.5cm]
&\,&  &=& \sum \limits_{n=2}^{\infty}\, n\, \left( 1-\frac{1}{k} \right)^{2} \left( \frac{1}{n^{k}}-\frac{1}{n^{k+1}} \right)  \\[.5cm]
&\,&  &\geq& \sum \limits_{n=2}^{\infty} \left( \frac{1}{n^{k}}-\frac{1}{n^{k+1}} \right)=\zeta(k)-\zeta(k+1).
\end{array}
$$
\end{proof}

\begin{lem}
Recall $B_k=\left| b_k \right|.$ For $k> 1,$ $B_k-B_{k-1}\leq \frac{1}{k^2}.$ 
\label{eq:diff_b_k_les}
\end{lem}

\begin{proof}
$$
\begin{array}{lcl}
&\,& \frac{1}{k^2}-\left( B_k-B_{k-1} \right) = \frac{1}{k^2}-B_k+B_{k-1}\\[.5cm]

&=& \frac{1}{k^2} -\left( 1-\frac{1}{2^k} \right)\zeta(k+1)+\left( 1-\frac{1}{2^{k-1}} \right)\zeta(k)\\[.5cm]
&=& \frac{1}{k^2} -\left( 1-\frac{1}{2^{k+1}} +\frac{1}{3^{k+1}}-\frac{1}{4^{k+1}}+\cdots \right)+\left( 1-\frac{1}{2^{k}} +\frac{1}{3^{k}}-\frac{1}{4^{k}}+\cdots \right) \\[.5cm]
&=& \frac{1}{k^2} -\frac{1}{2^k}\left( 1-\frac{1}{2} \right)+\frac{1}{3^k}\left( 1-\frac{1}{3} \right)-\frac{1}{4^k}\left( 1-\frac{1}{4} \right)+\cdots \\[.5cm]
&>& \frac{1}{k^2} -\frac{1}{2^k}\left( 1-\frac{1}{2} \right) > 0 \hspace{1.5cm} \text{for all $k> 1$.}
\end{array}
$$
\end{proof}

\section{Proofs of theorems of section 2}

\noindent \textbf{Proof of Theorem \ref{eq:th}\eqref{eq:O_1_k}}

\vspace{.3cm}

\begin{proof}
Using \eqref{eq:b_k} and Lemma \ref{eq:diff_zeta_k_les}, notice that for $k-j\geq 4$\\

\noindent $\left| \frac{b_{k-j}}{k+1}+\frac{b_{k-1-j}}{k} \right|$
\begin{dmath*}
=\left| \frac{(-1)^{k-j+1}\left( 1-2^{-k+j} \right)\zeta(k-j+1)}{k+1} + \frac{(-1)^{k-j}\left( 1-2^{-k+1+j} \right)\zeta(k-j)}{k}\right|
=\left| \frac{\left( 1-2^{-k+1+j} \right)\zeta(k-j)}{k}-\frac{\left( 1-2^{-k+j} \right)\zeta(k-j+1)}{k+1} \right|
=\frac{1}{k(k+1)}\left| (k+1)\left( 1-\frac{1}{2^{k-1-j}} \right)\zeta(k-j)-k\left( 1-\frac{1}{2^{k-j}} \right)\zeta(k-j+1) \right|
=\frac{1}{k(k+1)}\left| k\left( \zeta(k-j)-\zeta(k-j+1) \right)-\frac{k}{2^{k-j}}\left( 2\zeta(k-j)-\zeta(k-j+1) \right)+ \left( 1-\frac{1}{2^{k-1-j}} \right)\zeta(k-j) \right|
\leq \frac{1}{k(k+1)}\left[ k\left( \zeta(k-j)-\zeta(k-j+1) \right)+\frac{k}{2^{k-j}}\left\{\left( \zeta(k-j)-\zeta(k-j+1) \right) +\zeta(k-j)\right\}+\left( 1-\frac{1}{2^{k-1-j}} \right)\zeta(k-j)\right]
\leq \frac{1}{k(k+1)}\left[ \frac{k}{(k-j)^2}+\frac{k}{2^{k-j}}\left\{ \frac{1}{(k-j)^2}+\zeta(2) \right\}+\zeta(2)\right]
=\frac{1}{(k+1)(k-j)^2}+\frac{1}{2^{k-j}(k+1)(k-j)^2}+\frac{\zeta(2)}{2^{k-j}(k+1)}+\frac{\zeta(2)}{k(k+1)}
\leq\frac{1}{(k+1)(k-j)^2}+\frac{1}{(k+1)(k-j)^2}+\frac{\zeta(2)}{2^{k-j}(k+1)}+\frac{\zeta(2)}{k(k+1)}
=\frac{2}{(k+1)(k-j)^2}+\frac{\zeta(2)}{2^{k-j}(k+1)}+\frac{\zeta(2)}{k(k+1)}.
\end{dmath*}

\noindent Therefore, \\

\noindent $\left| a_{k+1}+a_k \right|$
\begin{dmath*}
=\left| \frac{1}{k+1}\sum_{j=0}^{k-1}a_j\, b_{k-j}+\frac{1}{k}\sum_{j=0}^{k-2}a_j\, b_{k-1-j} \right|
=\left| \frac{a_{k-1}\, b_1}{k+1}+\sum_{j=0}^{k-2}a_j\,\left( \frac{b_{k-j}}{k+1}+\frac{b_{k-1-j}}{k}\right) \right|
\leq \frac{1}{k+1}+\sum_{j=0}^{k-2}\left| \frac{b_{k-j}}{k+1}+\frac{b_{k-1-j}}{k} \right|  \hspace{1.5cm} \text{by Remark \eqref{eq:b_les_1} and Lemma \eqref{eq:a_les_1}}
\leq \frac{1}{k+1}+\sum_{j=0}^{k-4}\left| \frac{b_{k-j}}{k+1}+\frac{b_{k-1-j}}{k} \right|+2\cdot \frac{\max\left\{ \left| b_3 \right|,\left| b_2 \right| \right\}}{k}+2\cdot \frac{\max\left\{ \left| b_2 \right|,\left| b_1 \right| \right\}}{k}
\leq \frac{1}{k+1}+\sum_{j=0}^{k-4}\left[ \frac{2}{(k+1)}\cdot \frac{1}{(k-j)^2}+\frac{\zeta(2)}{(k+1)}\cdot \frac{1}{2^{k-j}}+\frac{\zeta(2)}{k(k+1)} \right]+\frac{4}{k}
\leq \frac{5}{k}+\frac{2}{k+1}\sum_{j=0}^{k-4}\frac{1}{(k-j)^2}+\frac{\zeta(2)}{k+1}\sum_{j=0}^{k-4}\frac{1}{2^{k-j}}+\frac{\zeta(2)}{k(k+1)}\sum_{j=0}^{k-4}1
=\frac{5}{k}+\frac{2}{k+1}\left( \frac{1}{4^2}+\frac{1}{5^2}+\cdots +\frac{1}{k^2} \right)+\frac{\zeta(2)}{k+1}\left( \frac{1}{2^4}+\frac{1}{2^5}+\cdots + \frac{1}{2^k} \right)+\frac{\zeta(2)(k-3)}{k(k+1)}
\leq \frac{5}{k}+\frac{2}{k+1}\cdot \zeta(2)+\frac{\zeta(2)}{k+1}\cdot \frac{1}{1-\frac{1}{2}}+\frac{\zeta(2)}{k+1}
= \frac{5}{k}+\frac{5\zeta(2)}{k+1}
\leq \frac{5}{k}\left( 1+\zeta(2) \right).\\
\end{dmath*}

\noindent Therefore for $\displaystyle k\geq4$, $\left| a_{k+1}+a_{k} \right|\leq \frac{5}{k}\left( 1+\zeta(2) \right)$ and so $a_{k+1}+a_k=\mathcal{O}(1/k).$

\end{proof}

\vspace{.3cm}

\noindent \textbf{Proof of Theorem \ref{eq:th}\eqref{eq:1_pi}}
\begin{proof}
By definition of Akatsuka zeta Mahler measure (see \cite{R1}), the generating function $f(s)$ of $a_k$'s is nothing but $Z(s,x-r)$ with $\left| r \right| =1.$  From \cite{R2} we know that for $\left| r \right| =1$ and $\left| s \right| <1,$ 

$$f(s):=\sum \limits_{k=0}^{\infty} a_k  \, s^k=Z(s,x-r)=\frac{\Gamma(s+1)}{\Gamma ^{2}\left( \frac{s}{2}+1 \right)}=\frac{4}{s}\frac{\Gamma(s)}{\Gamma^{2}\left( \frac{s}{2} \right)}.$$ 

\noindent Define $\displaystyle F(s):=1+\sum \limits_{k=1}^{\infty}(-1)^k\left(a_{k-1}+a_{k}\right)s^{k}.$ So, $F(s)=(1-s)f(-s).$\\

\noindent Notice that $$\lim \limits_{s\rightarrow 1^{-}}F(s)=\frac{-4}{\Gamma^{2}\left( -\frac{1}{2} \right)}\lim \limits_{s\rightarrow 1^{-}} (1-s)\Gamma(-s)=\frac{-4}{\Gamma^{2}\left( -\frac{1}{2} \right)}\lim \limits_{s\rightarrow -1} (1+s)\Gamma(s)=\frac{1}{\pi},$$ since $\displaystyle \lim \limits_{s\rightarrow -1} (1+s)\Gamma(s)=-1$ and $\sqrt{\pi}=\Gamma(1/2)=(-1/2) \, \Gamma( -1/2).$ \\

\noindent Now $\{ k(-1)^{k}\left(a_k+a_{k+1}\right) \}$ is a bounded sequence by Theorem \ref{eq:th}\eqref{eq:O_1_k}. Therefore applying Littlewood's extension of Tauber's Theorem (see \cite{tauber}) on the sequence $\{ (-1)^{k} \left( a_k+a_{k+1} \right) \}$ and its generating function $F(s)-1$ we see that 

$$\lim \limits_{k\rightarrow \infty} \left| a_k \right| =  1-\sum \limits_{k=0}^{\infty} \, \{ (-1)^{k} \left( a_k+a_{k+1} \right) \} = 1+\lim \limits_{s\rightarrow 1^{-}}\left( F(s)-1\right)=\frac{1}{\pi}.$$

\end{proof}

\vspace{.3cm}

\noindent \textbf{Proof of Theorem \ref{eq:th}\eqref{eq:o_1_k}}
\begin{proof}
Recall $B_k=\left| b_k \right|$ from Lemma \ref{eq:diff_b_k_les}. Now define a new sequence $\{A_k\}$ such that $A_0=1,$ $A_1=0$ and $$A_k=\frac{1}{k}\, \sum \limits_{j=0}^{k-2}A_{j}\, B_{k-1-j} $$ for all $k\geq 2.$ A careful observation of the individual terms inside $a_k$ and $A_k$ easily shows that $A_k=\left| a_k \right|$. Clearly $A_k=\left| a_k \right| \leq 1$ by Lemma \ref{eq:a_les_1}. Let $m:=\left \lfloor (k-2)/2 \right \rfloor$ and $A:=1/\pi.$ Since $\lim \limits_{k \rightarrow \infty} A_k=1/\pi=A$, using Remark \ref{eq:b_les_1} and Lemma \ref{eq:diff_b_k_les}, we see for each $\epsilon>0$ there is a sufficiently large integer $N>0 $ such that $k>N$ implies
\begin{eqnarray}
\left|(k+1) \left( a_{k+1}+a_k \right) \right|
&=& \left|(k+1) \left( A_{k+1}-A_k \right) \right| \nonumber \\
&=& \left| \sum \limits_{j=0}^{k-1}A_{j}B_{k-j}-\sum \limits_{j=0}^{k-2}A_{j}B_{k-1-j}-A_k \right|  \nonumber \\ \label{eq:absolute}
&\leq& \left| A_{k-1}B_1-A_k + \sum \limits_{j=m+1}^{k-2}A_j \left( B_{k-j}-B_{k-1-j} \right) \right|  \\ 
&\,& + \sum \limits_{j=0}^{m}A_j\left( B_{k-j}-B_{k-1-j} \right).  \nonumber
\end{eqnarray}

Now if the object within the absolute value signs in (\ref{eq:absolute}) is positive, then 
\begin{eqnarray} 
&\,&\left|(k+1) \left( a_{k+1}+a_k \right) \right|  \nonumber \\ \label{eq:absolute_epsilon}
&\leq& \left| \left( A+\epsilon \right) B_1-\left( A-\epsilon \right) + \left( A+\epsilon \right) \!\!\! \sum \limits_{j=m+1}^{k-2} \left( B_{k-j}-B_{k-1-j} \right) \right|\\
&\,& +\sum \limits_{j=0}^{m} \frac{A_j}{(k-j)^2}  \nonumber \\
&\leq& \left| \left( A+\epsilon \right) B_1-\left( A-\epsilon \right) +  \left( A+\epsilon \right) \left( B_{k-m-1}-B_{1} \right) \right| \nonumber \\
&\,& +\frac{1}{(k-m)^2}(m+1) \nonumber 
\end{eqnarray} 

\noindent Notice $ B_{k-m-1} \rightarrow 1$ and  $ (m+1)/(k-m)^2 \rightarrow 0$ when $k \rightarrow \infty.$  Therefore  we have $$\lim \limits_{k \to \infty} \left|(k+1) \left( a_{k+1}+a_k \right) \right| \leq \left| \left( A+\epsilon \right) B_1-\left( A-\epsilon \right) +  \left( A+\epsilon \right) \left(1-B_{1} \right) \right|.$$

\noindent Since the above inequality holds for each fixed $\epsilon>0,$ it also holds for $\epsilon = 0.$ Hence we have $ \left| (k+1) \left( a_{k+1}+a_k \right) \right|  \rightarrow 0$ when $k \rightarrow \infty.$ Therefore, $a_{k+1}+a_k=o\left( 1/k \right).$

But if the object within the absolute value signs in (\ref{eq:absolute}) is negative, then a similar argument gives the same conclusion just by replacing $+\epsilon$ by $-\epsilon$ in (\ref{eq:absolute_epsilon}).
\end{proof}

\vspace{.3cm}

\noindent \textbf{Proof of Theorem \ref{eq:th}\eqref{eq:-1}}
   
\begin{proof}
From Theorem \ref{eq:th}\eqref{eq:1_pi} we know that $0<  \lim \limits_{k\rightarrow \infty} \left| a_k \right| = 1/\pi < \infty$. Now using Remark \ref{eq:b_alternating} we have

$$\lim \limits_{k\rightarrow \infty}\frac{a_{k+1}}{a_k}=-1.$$
\end{proof}

\subsection*{Acknowledgments:}  We would like to thank George E. Andrews and Chris Monico for helpful discussions.

\end{document}